\documentclass[12pt]{amsart}
\usepackage{epsfig,amscd,amssymb,amsxtra,amsmath,amsthm,multicol,makecell}

\newcommand{\bbC}{{\mathbb{C}}}

\newcommand{\bbR}{{\mathbb{R}}}

\numberwithin{equation}{section}

\newtheorem{Lem}[equation]{Lemma}

\newtheorem{Thm}[equation] {Theorem}
\newtheorem{Cor}[equation]{Corollary}

\title
[On degrees  in family ]
{On degrees in family of maps constructed via modular forms}

\author{ Goran Mui\'c}
\address{
  Department of Mathematics,
  Faculty of Science,
University of Zagreb,
Bijeni\v cka 30, 10000 Zagreb,
Croatia}
 \email{gmuic@math.hr}

\begin{document}

\begin{abstract}
  This paper is a continuation of our previous works (see Mui\' c in Monatsh. Math. 180, no. 3, 607--629,  (2016)) and (Mui\' c, Kodrnja in Ramanujan J. 55, no. 2, 393–-420, (2021))  where we have
  studied maps from $X_0(N)$ into $\mathbb P^2$ (and more general) constructed via modular forms of the same weight.
  In this short note we  study how degrees of the maps and degrees of the resulting curve change when we let modular forms vary. 
\end{abstract}


\thanks{The  author acknowledges Croatian Science Foundation grant IP-2018-01-3628.}
\maketitle

\section{Introduction}\label{intr}

 In our earlier paper \cite{Muic2} we gave fairly general study  of complex holomorphic maps
 $X_0(N)\longrightarrow \mathbb P^2$ (and more general) and proved a formula for the degrees \cite[Theorem 1-4]{Muic2} described below in Theorem \ref{starithm}.
 Based on \cite[Theorem 1-4]{Muic2}, we   developed the test for  birationality of the maps (see the introduction in \cite{Muic2}).  Using these results, the problems of constructing birational maps
 into $\mathbb P^2$ has been studied in \cite{MuKo} with emphasis on the  explicit computations in SAGE.  The paper \cite{MuKo} constructs various models over
 $\mathbb C$ of $X_0(N)$ complementing previous works such as \cite{BKS}, \cite{bnmjk},
 \cite{sgal}, \cite{ishida}, \cite{Kodrnja1}, \cite{Muic}, \cite{MuMi}, \cite{mshi} and \cite{yy}. The purpose of the present short note is to study how degrees of the maps and degrees of the
 resulting curve change when we let modular forms vary.  The main result is Theorem \ref{mainthm} (see below).

We continue by recalling some standard facts from \cite{Miyake}.  Let $\mathbb H$ be the upper half--plane.
Then the group $SL_2(\bbR)$  acts on $\mathbb H$  as follows:
$$ g.z=\frac{az+b}{cz+d}, \ \ g=\left(\begin{matrix}a & b\\ c & d
\end{matrix}\right)\in SL_2(\bbR).
$$
We let $j(g, z)=cz+d$. The function $j$ satisfies the cocycle identity:
\begin{equation}\label{cocycle}
j(gg', z)=j(g, g'.z) j(g', z).
\end{equation}

Next, $SL_2(\bbR)$--invariant measure on $\mathbb H$ is defined by $dx dy
/y^2$, where the coordinates on $\mathbb H$ are written in a usual way 
$z=x+\sqrt{-1}y$, $y>0$. A discrete subgroup $\Gamma\subset
SL_2(\bbR)$ is called a Fuchsian group of the first kind if 
$$
\iint _{\Gamma \backslash \mathbb H} \frac{dxdy}{y^2}< \infty.
$$
Then, adding a finite number of points 
in $\bbR\cup \{\infty\}$ called cusps, $\mathcal F_\Gamma$ can be
compactified. In this way we obtain a compact Riemann surface 
$\mathfrak R_\Gamma$. Let $g(\Gamma)$ be the genus of $\mathfrak R_\Gamma$.  One of the most important examples are the groups
$$
\Gamma_0(N)=\left\{\left(\begin{matrix}a & b \\ c & d\end{matrix}\right) \in SL_2(\mathbb Z); \ \
    c\equiv 0 \ (\mathrm{mod} \ N)\right\}, \ \ N\ge 1.
$$
    We write $X_0(N)$ for    $\mathfrak R_{\Gamma_0(N)}$.  As usual we consider $\mathfrak R_\Gamma$ as a smooth irreducible projective curve over $\mathbb C$ with the field
    of rational functions $\mathbb C\left(\mathfrak R_\Gamma\right)$ to be the field of metomorphic functions on $\mathfrak R_\Gamma$.

Let $\Gamma$ be a Fuchsian group of the first kind. 
Let $\chi$ be a character $\Gamma \rightarrow\mathbb C^\times$
of finite order. Let $m\ge 1$. We consider the space $M_m(\Gamma, \chi)$
(resp., $S_m(\Gamma, \chi)$)  of all 
modular (resp., cuspidal) forms of weight $m$;  this is a space
of all holomorphic functions $f: \mathbb H\rightarrow \bbC$ 
such that $f(\gamma.z)=\chi(\gamma)j(\gamma, z)^m f(z)$ 
($z\in \mathbb H$, $\gamma\in \Gamma$) which are holomorphic (resp., holomorphic and vanish) at
every cusp for $\Gamma$. When $\chi$ is trivial, we write $M_m(\Gamma)$ and $S_m(\Gamma)$ instead of
$M_m(\Gamma, \chi)$ and $S_m(\Gamma, \chi)$, respectively.

Assume that $\dim M_m(\Gamma, \chi)\ge 3$.  
We select three linearly independent modular forms $f, g$, and $h$ in $M_m(\Gamma, \chi)$, and  
construct the  holomorphic map 
$\mathfrak R_\Gamma\longrightarrow \mathbb P^{2}$ which is uniquely determined by being initially defined  by
\begin{equation}\label{varphi-int}
z\longmapsto (f(z):  g(z): h(z))
\end{equation}
on the complement of a finite set of $\Gamma$--orbits in $\mathfrak R_\Gamma$ of common zeroes of $f, g$ and $h$.
The image is an irreducible projective plane curve, which we denote by $\mathcal C(f, g, h)$ (see \cite[Lemma 3-1]{Muic2}).
We denote by 
$$
\mathbb C\left(\mathcal C(f, g, h)\right)
$$
the field of rational functions on $\mathcal C(f, g, h)$. It can be realized as a subfield of the field 
$\mathbb C\left(\mathfrak R_\Gamma\right)$ of rational functions on $\mathfrak R_\Gamma$ generated over $\mathbb C$ by
$g/f$ and $h/f$. By the 
usual definition, the degree of the map 
(\ref{varphi-int}), denoted by
$$ 
d(f, g, h),
$$ is the degree of the field extension 
$$
\mathbb C\left(\mathcal C(f, g, h)\right) \subset 
\mathbb C\left(\mathfrak R_\Gamma\right).
$$

Let $l$ be a line in $\mathbb P^{2}$ in general position with respect 
to $\mathcal C(f, g,  h)$. Then, it intersects $\mathcal C(f, g,h)$ in different points the number of 
which is the degree of $\mathcal C(f, g,h)$. We denote the degree of   $\mathcal C(f, g,h)$ by $\deg{\mathcal C(f, g, h)}$.
The main result of \cite{Muic2} (see \cite[Theorem 1-4]{Muic2}) proves the following:

\begin{Thm}\label{starithm}  Assume that  $\dim M_m(\Gamma, \chi)\ge 3$. Assume that
  $f, g, h\in M_m(\Gamma, \chi)$ are linearly independent.
Then, we have the following:
$$
d(f, g, h) \cdot \deg{\mathcal C(f, g, h)}= \eta_m(f, g, h),
$$
where
$$
\eta_m(f, g, h)=\frac{m}{4\pi} \iint _{\Gamma \backslash \mathbb H} \frac{dxdy}{y^2}
- \sum_{\mathfrak a\in \mathfrak R_\Gamma} 
\min{\left(\nu_{\mathfrak a}(f), \nu_{\mathfrak a}(g), \nu_{\mathfrak a}(h)\right)}.
$$ 
\end{Thm}

Now, we discuss the main result of the present paper. We introduce some notation.  Let $\Xi\subset M_m(\Gamma, \chi)$ be a subspace such that $\dim{\Xi}\ge 3$.
Assume that   $f, g \in \Xi$ are linearly independent.
For each $l\ge 1$ (resp., $k\ge 1$), we let $X_l$ (resp., $Z_k$) be the set of all
$h\in \Xi - \left(\mathbb C f+ \mathbb C g\right)$ such that $\deg{\mathcal C(f, g, h)}=l$ (resp., $d(f, g, h)=k$). By Theorem \ref{starithm}, we have
$$
\text{$X_l$ and $Z_k$ are  empty for $k, l>  \frac{m}{4\pi} \iint _{\Gamma \backslash \mathbb H} \frac{dxdy}{y^2}$}.
$$
Obviously, we have
$$
\Xi- \left(\mathbb C f+ \mathbb C g\right)=\cup_{l\ge 1}\ X_l=  \cup_{k\ge 1} \ Z_k.
$$
We remark that  $X_1=\emptyset$ since $f, g$, and $h$ are linearly independent.
The main result of the present paper is the following theorem:

\begin{Thm}\label{mainthm}  Let $\Xi\subset M_m(\Gamma, \chi)$ be a subspace such that $\dim{\Xi}\ge 3$. We equip $\Xi$ with Zariski topology. 
Assume that   $f, g \in \Xi$ are linearly independent.  Then,  we have the following:
  \begin{itemize}
  \item[(i)]   The sets $X_l$ are locally closed. The set of all $h\in \Xi- \left(\mathbb C f+ \mathbb C g\right)$ such that
    $\deg{\mathcal C(f, g, h)}$ is largest possible is open (i.e., the largest $L$ such that $X_L\neq \emptyset$).
  \item[(ii)]  For each $l$ such that $X_l\neq \emptyset$, the set of all $h\in X_l$ such that $d(f, g, h)$ is smallest possible in $X_l$ is  an open set in $X_l$.
  \item[(iii)] The sets $Z_k$ are constructible. The set  $Z_k$ is empty set unless $k$ divides
  $\left[\mathbb C(\mathfrak R_\Gamma): \mathbb C(g/f)\right]$. 
  \end{itemize}
  \end{Thm}

\vskip .2in
Recall that  we say that  $\Xi$ {\bf determines  the field of rational functions}  $\mathbb C(\mathfrak R_\Gamma)$ if  
there exists  a basis  $f_0, \ldots, f_{s-1}$ of $W$,  such that $\mathbb C(\mathfrak R_\Gamma)$ is generated over $\mathbb C$ by the
quotients $f_i/f_0$, $1\le i\le s-1$  (see \cite[Definition 1.3]{MuKo}). The notion is independent of the basis. The introduction of \cite{MuKo}
contains many examples of such spaces $\Xi$ (called $W$ there). For example, we may take $\Xi=S_2(\Gamma)$ if $\Gamma$ is not hyperelliptic (\cite{Ogg} has determined  all $\Gamma_0(N)$
such that $X_0(N)$ is not hyperelliptic (implies $g(\Gamma_0(N))\ge 3$)). Also,  for $m\ge 4$ is even,  if $\dim S_m(\Gamma) \ge \max{(g(\Gamma)+2, 3)}$, then we can take $\Xi=S_m(\Gamma)$
by general theory of algebraic curves \cite[Corollary 3.4]{Muic1}.

\begin{Cor}\label{maincor}  Maintaining assumptions of Theorem \ref{mainthm}, we assume also that  $\Xi$  determines  the field of rational functions  $\mathbb C(\mathfrak R_\Gamma)$. Let
  $L$ be the largest possible such that $X_L\neq 0$. Then, $Z_1\cap X_L$ contains a non--empty open set. In other words, the set of all $h\in \Xi- \left(\mathbb C f+ \mathbb C g\right)$ such that
    $\deg{\mathcal C(f, g, h)}$ is largest possible and $d(f, g, h)=1$ is non--empty open set. 
\end{Cor}
\begin{proof} By \cite[Theorem 1.4]{MuKo}, $Z_1$ contains a non--empty open set. On the other hand, by Theorem \ref{mainthm}, $X_L$ is open. The corollary follows. 
  \end{proof}

The corollary improves on  \cite[Theorem 1.4]{MuKo} since in the language of the present paper there we could just construct open set in $\Xi- \left(\mathbb C f+ \mathbb C g\right)$
such that $d(f, g, h)=1$ without control of  $\deg{\mathcal C(f, g, h)}$. The corollary also generalizes   \cite[Corollary 3.7]{MuKo}
with a similar conclusion but with more restrictive assumptions.

\section{The Proof of Theorem \ref{mainthm}}
Let  $h\in \Xi- \left(\mathbb C f+ \mathbb C g\right)$. Let $P_{f, g, h}$ be an irreducible homogeneous polynomial which locus
is $C(f, g, h)$. Equivalently,   $P_{f, g, h}(f(z), g(z), h(z))=0$, for all $z\in \mathbb H$.  The polynomial $P_{f, g, h}$ is unique up to a multiplication by a non–zero constant.
The dehomogenization $Q_{f, g, h}$ of $P_{f, g, h}$ with respect to the last variable satisifies $Q_{f, g, h}(g/f, h/f)=0$ in the field of rational functions $\mathbb C(\mathfrak R_\Gamma)$.
It is very easy to check that $Q_{f, g, h}(g/f, \cdot)$ is irreducible as a polynomial with coefficients in the field  $\mathbb C(g/f)$. Thus, it is a minimal polynomial of $h/f$ over  $\mathbb C(g/f)$.
Hence, it is equal to the degree of the   field extension $\mathbb C(g/f)\subset \mathbb C(\mathcal C(f, g, h))$
$$
\left[\mathcal C(f, g, h): \mathbb C(g/f)\right] = \deg{Q_{f, g, h}(g/f, \cdot)}.
$$
If we consider the field extensions
$$
\mathbb C(g/f)\subset \mathbb C(\mathcal C(f, g, h)) \subset \mathbb C(\mathfrak R_\Gamma),
$$
and compute their degrees, then we obtain the next lemma.

\begin{Lem}\label{l-1} For $h\in \Xi- \left(\mathbb C f+ \mathbb C g\right)$, the product  $\deg{Q_{f, g, h}(g/f, \cdot)}\cdot d(f, g, h)$ does not depend on $h$. It is equal to
  the degree $\left[\mathbb C(\mathfrak R_\Gamma): \mathbb C(g/f)\right]$ (i.e, to the degree of divisor of poles of $g/f$). In particular, $Z_k$ is empty set unless $k$ divides
  $\left[\mathbb C(\mathfrak R_\Gamma): \mathbb C(g/f)\right]$. 
   \end{Lem}

We continue with the following two lemmas:

\begin{Lem}\label{l-2}  For each $k\ge 1$,   let $Y_k$ be  the set of all $h\in  \Xi- \left(\mathbb C f+ \mathbb C g\right)$  such that $\deg{Q_{f, g, h}(g/f, \cdot)}=k$.
  Then, $Y_k$ is empty unless $k$ divides
  $\left[\mathbb C(\mathfrak R_\Gamma): \mathbb C(g/f)\right]$. If this is so, we have
  $Z_k=Y_{\left[\mathbb C(\mathfrak R_\Gamma): \mathbb C(g/f)\right]/k}$.
\end{Lem}
\begin{proof}  This follows immediately from Lemma \ref{l-1}. 
\end{proof}

\begin{Lem}\label{l-3}  For each $l\ge 1$,   $X_l$ is  the set of all $h\in  \Xi- \left(\mathbb C f+ \mathbb C g\right)$  such that $\deg{P_{f, g, h}}=l$
\end{Lem}
\begin{proof}  It is well--known and easy to check directly that  $\deg{\mathcal C(f, g, h)}= \deg{P_{f, g, h}}$.
  \end{proof}

\begin{Lem}\label{l-4}  For  all $h\in  \Xi- \left(\mathbb C f+ \mathbb C g\right)$, we have $\deg{Q_{f, g, h}(g/f, \cdot)} \le \deg{P_{f, g, h}}\le \eta_m(f, g, h)\le
  \frac{m}{4\pi} \iint _{\Gamma \backslash \mathbb H} \frac{dxdy}{y^2}$.
  \end{Lem}
\begin{proof}  This follows from Theorem \ref{starithm} and the fact that $\deg{\mathcal C(f, g, h)}= \deg{P_{f, g, h}}$. 
  \end{proof}

Now, we prove the key lemma. 
  
\begin{Lem}\label{l-5} We have the following: 
\begin{itemize}
\item[(i)] The sets $X_l$ are locally closed, $X_1=\emptyset$, and the set of  all $\Xi- \left(\mathbb C f+ \mathbb C g\right)$ such that
    $\deg{P_{f, g, h}}$ is largest possible is open (well--defined because of Lemma \ref{l-4}).
\item[(ii)] For each $l$ such that $X_l\neq \emptyset$, the set of all $h\in X_l$ such that $\deg{Q_{f, g, h}}$ is largest possible in $ X_l$  is  an open set in $X_l$.
\item[(iii)] The sets $Y_k\cap X_l$ are locally closed, and $Y_k$ are constructible sets for all $k, l$.
  \end{itemize}
  \end{Lem}
\begin{proof}  We let
 $h=\lambda_0 f+ \lambda_1 g+\lambda_2 f_2+\cdots + \lambda_{s-1} f_{s-1}\in \Xi- \left(\mathbb C f+ \mathbb C g\right)$.
  Let $l\ge 1$ be an integer. For  $\alpha=(\alpha_0, \alpha_1, \alpha_2)\in \mathbb Z^3_{\ge 0}$
such that $|\alpha|\overset{def}{=}\alpha_0+ \alpha_1 +\alpha_2=l$, and
$(i_0, i_1, \ldots, i_{s-1}) \in \mathbb Z^s_{\ge 0}$ such that $\sum_{j=0}^{s-1} i_j=\alpha_2$, we consider a cusp form 
$$
f^{\alpha_0+i_0}g^{\alpha_1+i_1} f_2^{i_2}\cdots f_{s-1}^{i_{s-1}} =\sum_{n=1}^\infty \ b_n(\alpha, i_0, i_1, \ldots, i_{s-1}) q^n \in S_{lm}(\Gamma).
$$
We define homogeneous polynomials for all $n\ge 1$ as follows: 
$$
  B_{\alpha, n} \left(\lambda_0, \ldots,\lambda_{s-1}\right) =\sum_{i_0+i_1+\cdots + i_{s-1}=\alpha_2} \binom{\alpha_2}{i_0, i_1, \ldots, i_{s-1}} b_n(\alpha, i_0, i_1, \ldots, i_{s-1})
  \lambda_0^{i_0}\cdots \lambda_{s-1}^{i_{s-1}}
  $$
 We order all $\binom{l+2}{2}$ $\alpha$'s in the lexicographical order: $(0, 0, l)< (0, 1, l-1)< \cdots < (l, 0, 0)$.
 Then, we can consider vectors:
  $$
 C_n\left(\lambda_0, \ldots,\lambda_{s-1}\right) \overset{def}{=}\left(B_{(0, 0, l), n} \left(\lambda_0, \ldots,\lambda_{s-1}\right), \ldots, B_{(l, 0, 0), n} \left(\lambda_0, \ldots,\lambda_{s-1}\right)\right) \in
 \mathbb C^{\binom{l+2}{2}}.
  $$
  Let $C\left(\lambda_0, \ldots,\lambda_{s-1}\right)$ be an infinite matrix with rows $C_n\left(\lambda_0, \ldots,\lambda_{s-1}\right)$. Next, the reader can easily check that
  $(a_{(0, 0, l)}, \ldots, a_{(l, 0, 0)})\in \mathbb C^{\binom{l+2}{2}}$ is the solution of  the system of homogeneous equations: \footnote{ By using Sturm bound for $S_{lm}(\Gamma)$ we can bound the number of
  equation, but this is not important here.}  
  \begin{equation} \label{system}
  \sum_\alpha \ B_{\alpha, n} \left(\lambda_0, \ldots,\lambda_{s-1}\right) a_\alpha= 0, \ \ n\ge 1,
  \end{equation}
  if and only if
  $$
 \sum_{\alpha}   a_\alpha  f(z)^{\alpha_0}g(z)^{\alpha_1} h(z)^{\alpha_2}=0, \ \ z\in \mathbb H.
  $$
 Equivalently, the corresponding homogeneous polynomial belongs to the ideal of the curve $\mathcal C(f, g, h)$. This implies that the system (\ref{system}) has only a trivial solution if
 $l< \deg{P_{f, g, h}}$ while for $l\ge \deg{P_{f, g, h}}$ there always exist a non--trivial solution. The solution is unique up to a multiplication by a non--zero element in $\mathbb C$ if
 $l= \deg{P_{f, g, h}}$, and determines $P_{f, g, h}$. On the other hand, by Linear algebra, the system (\ref{system}) has a zero solution if and only if $\mathbb C$--span of vectors
 $C_n\left(\lambda_0, \ldots,\lambda_{s-1}\right)$, $n\ge 1$, is whole $\mathbb C^{\binom{l+2}{2}}$. Equivalently, there is a minor of $C\left(\lambda_0, \ldots,\lambda_{s-1}\right)$ of
 size $\binom{l+2}{2} \times \binom{l+2}{2}$ which is non--zero. Similarly, when $l= \deg{P_{f, g, h}}$,  $\mathbb C$--span of vectors
 $C_n\left(\lambda_0, \ldots,\lambda_{s-1}\right)$, $n\ge 1$, must be of codimension one in $\mathbb C^{\binom{l+2}{2}}$. Equivalently, there must exists a non--zero minor of
 $C\left(\lambda_0, \ldots,\lambda_{s-1}\right)$ of size $\binom{l+2}{2}-1 \times \binom{l+2}{2}-1$ while all minors of size 
 size $\binom{l+2}{2} \times \binom{l+2}{2}$ must be equal to zero.

 Above discussion implies the first claim in (i) i.e., the sets $X_l$ are locally closed. $X_1=\emptyset$ since $f, g$, and $h$ are linearly independent.
 Let $L$ be the largest possible degree of $\deg{P_{f, g, h}}$ when $h$ ranges over
 $\Xi- \left(\mathbb C f+ \mathbb C g\right)$.
 We prove that  $X_L$ is open. First we note that all 
 $\binom{L +2}{2} \times \binom{L+2}{2}$ minors are identically equal to zero  on $\mathbb C^s$. Indeed, if there would be a non--identically zero minor, say $M$,
 of that size, then the system (\ref{system})
 has a trivial solution when $M(\lambda_0, \ldots, \lambda_{s-1})\neq 0$.
 Hence, $\deg{P_{f, g, h}}> L$ for $h$ such that $M(\lambda_0, \ldots, \lambda_{s-1})\neq 0$ which is a contradiction. Now, $X_L$ is open since $h\in
 X_L$ if and only if there exists a minor $M$ of size $\binom{L +2}{2}-1 \times \binom{L+2}{2}-1$  such that $M(\lambda_0, \ldots, \lambda_{s-1})\neq 0$. This completes the proof of (i).

 Now, we prove (ii). Assume that $l$ satisfies $X_l\neq \emptyset$. Let $L$ be the largest possible degree of $\deg{Q_{f, g, h}}$ when $h$ ranges over $X_l$. Let $h\in X_l$.
 Then, the solution of the corresponding system
 (\ref{system}) satifies $a_\beta=0$ for all $\beta$ of the form $\beta=(\beta_0, \beta_1, L')$ with $L'>L$. Of course, for  $h\in X_l$ 
 such that $a_\beta\neq 0$, for some $\beta$ of the form $(\beta_0, \beta_1, L)$, we have  $\deg{Q_{f, g, h}}=L$.

Open sets $M(\lambda_0, \ldots, \lambda_{s-1})\neq 0$, where $M$ ranges over all minors
 $\binom{l+2}{2}-1 \times \binom{l+2}{2}-1$ cover $X_l$. Let us fix such minor $M$. Then, $M$ is obtained by removing some column (and using  $\binom{l+2}{2}-1$ rows but they are not important here),
 say $\gamma$-th. Then, rewritting the system (\ref{system}) in the form
 $$
  \sum_{\alpha\neq \gamma} \ B_{\alpha, n} \left(\lambda_0, \ldots,\lambda_{s-1}\right) a_\alpha=  - B_{\gamma, n} \left(\lambda_0, \ldots,\lambda_{s-1}\right) a_\gamma, \ \ n\ge 1,
$$
we see that the solution is given by
\begin{equation} \label{system-1}
  a_\alpha= - a_\gamma \frac{M^\alpha\left(\lambda_0, \ldots,\lambda_{s-1}\right)}{M\left(\lambda_0, \ldots,\lambda_{s-1}\right)}, \ \ \text{for all} \ \ \alpha\neq a_\gamma,
\end{equation}
where the determinant $M^\alpha$ is obtained from $M$ by replacing $\alpha$-th column with the column of the corresponding $B_{\gamma, n}$'s (i.e., using $n$'s that determine rows  of $M$).
Using (\ref{system-1}), we see that $a_\gamma\neq 0$. Also, we see that $a_\alpha\neq 0$ if and only if  $M^\alpha\left(\lambda_0, \ldots,\lambda_{s-1}\right)\neq 0$. Let us fix now arbitrary $\alpha$. Then,
 letting  $M$ vary, we see that the set of all
 $h\in X_l$ such that $a_\alpha\neq 0$ is open in $X_l$. This immediately implies that the set of all $h$ such that one of the coefficients $a_\beta$, for  $\beta$ of the form
 $(\beta_0, \beta_1, L)$, is non--zero is open. This completes the proof of (ii). 

 For the claim (iii), we recall that a constructible set is a finite union of locally closed sets. Then, it is enough to prove that
 $Y_k\cap X_l$ is locally closed for each $l$. As in the previous part of the proof, this intersection corresponds to the solution of the system (\ref{system}) such that $a_\alpha=0$, for all
 $\alpha$ of the form $\alpha=(\alpha_0, \alpha_1, k')$ with $k'>k$, and there exists $\beta$ such that $a_\beta\neq 0$ where $\beta=(\beta_0, \beta_1, k)$. As before, using (\ref{system-1}) this set is
 intersection of one closed set, one open set and $X_l$. But since $X_l$ is itself intersection of a closed and an open set, see the same holds for $Y_k\cap X_l$. This means that  $Y_k\cap X_l$
 is locally closed. This completes the proof of the lemma. 
\end{proof}

\vskip .2in
Finally, we complete the proof of Theorem \ref{mainthm}. (i) follows from Lemma \ref{l-5} (i) and Lemma \ref{l-3}. (ii) follows from Lemma \ref{l-5} (ii) and Lemma \ref{l-2}. The first claim in
(iii) follows also from Lemma \ref{l-2} and Lemma \ref{l-5} (iii). The second claim of (iii) is contained in Lemma \ref{l-1}.

\end{document}